\documentclass[a4paper,12pt]{amsart}
\usepackage[margin=3cm]{geometry}

\usepackage{amsthm,amsmath,amssymb,url,comment,setspace,mleftright}
\usepackage[dvipdfm]{graphicx,color}

\allowdisplaybreaks 

\theoremstyle{plain} 
\newtheorem{df}{Definition}[section]

\newtheorem{thm}[df]{Theorem}

\newtheorem{cor}[df]{Corollary}
\newtheorem{lem}[df]{Lemma}

\theoremstyle{definition}

\theoremstyle{remark}

\makeatletter

\@addtoreset{equation}{section}
\makeatother

\newcommand{\N}{\mathbb{N}}
\newcommand{\R}{\mathbb{R}}
\newcommand{\Z}{\mathbb{Z}}

\newcommand{\supp}{\mathrm{supp}}

\newcommand{\E}{\mathbb{E}}
\newcommand{\var}{\mathrm{Var}}

\newcommand{\1}[1]{\mathbf{1}_{#1}}

\newcommand{\half}{\frac{1}{2}}

\renewcommand{\tilde}{\widetilde}

\newcommand{\hyphen}{\textrm{-}}
\newcommand{\as}{\textrm{a.s.}}

\begin{document}
\title[First passage percolation]
{Upper bounds on the non-random fluctuations in first passage percolation with low moment conditions}
\author[N. Kubota]{Naoki KUBOTA}
\address{College of Science and Technology, Nihon University, Tokyo 101-8308, JAPAN}
\email{kubota.naoki08@nihon-u.ac.jp}
\keywords{First passage percolation, convergence rate}
\subjclass[2010]{60K35, 82B20}

\begin{abstract}
We consider first passage percolation with i.i.d.\,weights on edges of
the $d$-dimensional cubic lattice $\Z^d$.
Under the assumptions that
a weight is equal to zero with probability smaller than the critical probability
of bond percolation in $\Z^d$,
and has the $\alpha$-th moment for some $\alpha>1$,
we investigate upper bounds on the so-called non-random fluctuations of the model.
In addition, we give an application of our result to a lower bound for variance
of the first passage percolation in the case where the limit shape has flat edges.
\end{abstract}

\maketitle

\section{Introduction}\label{sec:intro}

\subsection{The model and the main result}
First passage percolation was originally introduced in 1965 by Hammersley and Welsh~\cite{HamWel65}.
In this model, we place i.i.d.~random weights on edges of the $d$-dimensional cubic lattice $\Z^d$,
and consider the minimum (random) traveling time from a subset of $\Z^d$ to another one.
Let $\mathcal{E}$ be the edge set of $\Z^d$
and consider the measurable space $\Omega:=[0,\infty)^{\mathcal{E}}$
endowed with the canonical $\sigma$-field $\mathcal{G}$.
Moreover, for a given probability measure $\nu$ on $[0,\infty)$,
let $P:=\nu^{\otimes \mathcal{E}}$ be the corresponding product measure
on $(\Omega,\mathcal{G})$.
For a nearest neighbor path $\gamma=(\gamma_0,\dots,\gamma_l)$ on $\Z^d$,
we define the \textit{passage time} of $\gamma$ as
\begin{align*}
 T(\gamma ):=\sum_{i=0}^{l-1} \omega (\{ \gamma_i ,\gamma_{i+1} \})
\end{align*}
with the convention $\sum_{i=0}^{-1} \omega (\{ \gamma_i ,\gamma_{i+1} \}):=0$.
Here we use the notation $\{ x,y \}$ to denote the edge of $\Z^d$ with endpoints $x$ and $y$.
For any two subsets $A$ and $B$ of $\Z^d$ we define the \textit{first passage time} from $A$ to $B$ as
\begin{align*}
 T(A,B):=\inf \biggl\{ T(\gamma );
 \begin{minipage}{6.8cm}
  $\gamma$ is a nearest neighbor path on $\Z^d$\\
  from some site in $A$ to some site in $B$
 \end{minipage}
 \biggr\}.
\end{align*}
In particular, write $T(x,y)=T(\{ x \}, \{ y \})$ for $x,y \in \Z^d$.
We may extend the first passage time over $\R^d$.
For $x \in \R^d$, 
let $[x]$ be a lattice point such that
\begin{align*}
 \|[x]-x\|_\infty =\min \bigl\{ \| v-x \|_\infty ;v \in \Z^d \bigr\} \leq \half,
\end{align*}
where $\| \cdot \|_\infty$ is the $\ell_\infty$-norm.
If $x$ and $y$ are in $\R^d$,
we rewrite $T(x,y):=T([x],[y])$.
To shorten notation, given a vector $\xi \in \R^d$,
the first passage time from the origin $0$ to $n\xi$ is denoted by
\begin{align*}
 a_{0,n}(\xi ):=T(0,n\xi ).
\end{align*}
It is well known from the standard subadditive ergodic theorem that
if $E[\omega(e)]<\infty$, then
for any $\xi \in \Z^d$, $P \hyphen \as$ and in $L^1$,
\begin{align}
 \mu(\xi)
 =\lim_{n \to \infty}\frac{1}{n}a_{0,n}(\xi )
 = \lim_{n \to \infty}\frac{1}{n}E[a_{0,n}(\xi )]
 = \inf_{n \geq 1}\frac{1}{n}E[a_{0,n}(\xi )].
 \label{eq:timeconst}
\end{align}
From \cite[pages~158-160]{Kes86_lect}, such a limit also exists for a general $\xi \in \R^d$,
and we call $\mu(\xi)$ the \textit{time constant} for $\xi \in \R^d$.

In this paper, we study rates of convergence to the time constant in the first passage percolation.
Kesten~\cite[(3.2), page~317]{Kes93} derived a bound on the so-called \emph{non-random fluctuations}
in first passage percolation, i.e.,
there exists a constant $C>0$ such that
\begin{align}\label{eq:kesnon}
 E[a_{0,n}(\xi )] -n\mu (\xi ) \leq Cn^{1-1/(2d+4)}(\log n)^{1/(d+2)},\qquad \xi \in \R^d,
\end{align}
under the assumptions that
\begin{align}\label{eq:p_c}
 \nu(\{ 0 \}) <p_c
\end{align}
where $p_c$ is the critical probability of bond percolation in $\Z^d$, and
\begin{align}\label{eq:exp_ass}
 \E[e^{\alpha \omega(e)}]<\infty \text{ for some } \alpha>0.
\end{align}
Alexander~\cite{Ale97} improved \eqref{eq:kesnon} by a different method.
On the other hand, Zhang~\cite{Zha10} studied the same problem under a weaker moment condition
than \eqref{eq:exp_ass}:
If
\begin{align}\label{eq:lowmnt}
 m_{\nu,\alpha}:=E[\omega(e)^\alpha]<\infty \text{ for some } \alpha>1,
\end{align}
then there exists a constant $C>0$ such that
for each coordinate direction $\xi'$ of $\R^d$,
\begin{align} \label{eq:zha10}
 E[a_{0,n}(\xi')] -n\mu (\xi') \leq Cn^{1/2}(\log n)^7.
\end{align}
For the proof of \eqref{eq:zha10}, he used symmetry properties of $\Z^d$
with respect to the coordinate axis.
Therefore, his approach does not work for any direction except coordinate axis,
and we need a new method.
The next theorem is our main result.

\begin{thm}\label{thm:rate}
Assume \eqref{eq:p_c} and \eqref{eq:lowmnt}.
Then, there exists a constant $C>0$ such that
for all $\ell_2$-unit vector $\xi \in \R^d$,
\begin{align}
 E[a_{0,n}(\xi )] -n\mu (\xi ) \leq Cn^{1-1/(6d+12)}(\log n)^{1/3}.
 \label{eq:rate}
\end{align}
\end{thm}

\subsection{Application of Theorem~\ref{thm:rate}}\label{sec:appli}
In this subsection, we state an application of Theorem~\ref{thm:rate}.
Bound \eqref{eq:rate} may not be optimal,
but it is very useful that for all direction $\xi$
we can uniformly take the exponent of the convergence rate strictly smaller than $1$.
Auffinger and Damron~\cite[Theorem~2.5]{AufDam13} established that
the variance of the first passage time has a lower bound with a logarithmic order
in the case where the limit shape has flat edges.
For Theorem~2.5 of \cite{AufDam13},
they require not only \eqref{eq:lowmnt} with $\alpha=2$ but also a bound on the non-random fluctuations
at that time.
Thanks to Theorem~\ref{thm:rate}, we can check their condition
whereas \eqref{eq:lowmnt} holds for $\alpha=2$.

Let $d=2$ and
write $\supp (\nu')$ for the support of the probability measure $\nu'$.
Moreover, let $\vec{p}_c$ be the critical parameter for oriented percolation on $\Z^2$.
Furthermore, denote by $\theta_q$ the unique angle such that the line segment connecting $0$
and the point $N_q:=(1/2+\alpha_q/\sqrt{2},1/2-\alpha_q/\sqrt{2}) \in \R^2$
has angle $\theta_q$ with the $x$-axis,
where $\alpha_q$ is the asymptotic speed of oriented percolation with parameter $q$.
For details of oriented percolation, we refer the reader to \cite{Dur84}.
For $q \geq \vec{p}_c$,
$\mathcal{M}_q$ is defined by the set of probability measures $\nu'$
satisfying conditions
\begin{enumerate}
\item[\bf (C1)] $\supp (\nu') \subset [1,\infty)$,
\item[\bf (C2)] $\nu' (\{ 1 \})=q$.
\end{enumerate}
Note that if $\nu \in \mathcal{M}_q$ (in particular, (C1) holds for $\nu$),
then we have $\nu (\{ 0 \})=0<p_c$, i.e., \eqref{eq:p_c} is satisfied.

We now assume that \eqref{eq:lowmnt} holds for $\alpha=2$ and the law $\nu$ satisfies one of conditions
\begin{itemize}
 \item[\bf (a)]
  $\inf \supp (\nu)=0$ and $\nu(\{ 0 \})<p_c$,
 \item[\bf (b)]
  $\lambda:= \inf \supp (\nu)>0$ and $\nu(\{ \lambda \})<\vec{p}_c$.
\end{itemize}
In \cite[Theorem~2]{NewPiz95}, under the above assumptions Newman and Piza showed that
there is a constant $C>0$ such that for all $n \geq 1$ and $\theta \in [0,2\pi)$,
\begin{align}\label{eq:aufdam_cor}
 \var (T(0,n\xi_\theta )) \geq C\log n,
\end{align}
where $\xi_\theta:=(\cos \theta, \sin \theta) \in \R^2$.
This means that the variance of the first passage time diverges as $n \to \infty$ in these cases.
On page~980 of \cite{NewPiz95},
they also state that the variance does not diverge for $\theta \in (\theta_q,\pi/2-\theta_q)$
in the case $\nu \in \mathcal{M}_q$ with $q>\vec{p}_c$.
We are now concerned with the divergence of $\var (T(0,n\xi_\theta ))$
for $\theta \in [0,\theta_q)$ in the same situation.
If $\xi_\theta$ is a coordinate direction, then
Zhang~\cite[Theorem~2]{Zha08} proved \eqref{eq:aufdam_cor} under assumption \eqref{eq:exp_ass}.
After that, Auffinger and Damron \cite[Theorem~2.5]{AufDam13} improved it as follows.

\begin{thm}[Auffinger and Damron] \label{thm:aufdam}
For a given $q \in [\vec{p}_c,1)$,
let $\nu \in \mathcal{M}_q$
and $\theta \in [0,\theta_q)$.
Suppose that \eqref{eq:lowmnt} holds with $\alpha=2$ and
there exists $\beta <1$ such that for all large $n$,
\begin{align}
 E[T(0,n\xi_\theta )]<n\mu (\xi_\theta )+n^\beta,
 \label{eq:aufdam_rate}
\end{align}
where $\xi_\theta:=(\cos \theta, \sin \theta) \in \R^2$.
Then, there exists a positive constant $C=C(\theta)$ such that
\eqref{eq:aufdam_cor} holds for all $n$.
\end{thm}

If we assume \eqref{eq:exp_ass}, then \eqref{eq:kesnon} yields
\eqref{eq:aufdam_rate} for all angles $\theta$,
and \eqref{eq:aufdam_cor} holds for all $\theta \in [0,\theta_q)$.
Under the assumption of Theorem~\ref{thm:aufdam}, \eqref{eq:zha10} only guarantees
the validity of \eqref{eq:aufdam_rate} for each coordinate direction $\xi_\theta$.
We use Theorem~\ref{thm:rate} to obtain \eqref{eq:aufdam_rate} for all angles $\theta$.
With these observations, the whole picture of divergence for $\var(T(0,n\xi_\theta))$
is completed under \eqref{eq:lowmnt} with $\alpha=2$.

\begin{cor}
For a given $q \in [\vec{p}_c,1)$,
let $\nu \in \mathcal{M}_q$ and $\theta \in [0,\theta_q)$.
Suppose that \eqref{eq:lowmnt} holds with $\alpha=2$.
Then, \eqref{eq:aufdam_cor} holds for all $n$.
\end{cor}

\subsection{Organization of the paper}\label{subsect:org}
Let us describe how the present article is organized.
In Section~\ref{sec:prelim}, we introduce truncated weights following the method of Zhang~\cite{Zha10}.
Since the argument of Sections~2 and 3 in \cite{Zha10} contains an oversight,
we will present one of ways to fix this (see Lemma~\ref{lem:zhang} below).
In addition, we give a method to compare the expectation of the first passage time for the truncated weights
with that for the original weights.

In Section~\ref{sec:pf_rate}, we give the proof of Theorem~\ref{thm:rate}.
To do this, we improve the approach taken in \cite[Section~3, page~317]{Kes93}
under our assumption \eqref{eq:lowmnt}.
This section is divided into two subsections.
In Subsection~\ref{subsect:kesap}, for the reader's convenience,
we explain the outline of Kesten's approach under \eqref{eq:exp_ass},
and clarify differences between his and ours.
In Subsection~\ref{subsect:fix}, we present a new method
to derive the convergence rate for all directions under low moment conditions.

In the following sections, $C_i$, $i=1,2,\dots$, are always positive constants
depending on $d$, $\nu$ and $\alpha$.

\section{Preliminaries}\label{sec:prelim}
In this section,
we shall introduce truncated weights, following basically the strategy taken in \cite{Zha10}.
By assumptions \eqref{eq:p_c} and \eqref{eq:lowmnt}, we can take $\kappa \in (0,1)$ such that
\begin{align*}
 P(\omega (e)<\kappa ) \vee P(\omega (e)>\kappa^{-1})<p_c.
\end{align*}
From now on, we fix $\kappa$ as above.
Then, an edge $e \in \mathcal{E}$ is said to be \textit{bad}
if $\omega(e)<\kappa$, and a site $x \in \Z^d$ is said to be \textit{unhealthy}
if some weights of $2d$ adjacent edges of $x$ are larger than $\kappa^{-1}$.
Let us now introduce two connectivities of paths on $\Z^d$.
We say that a path $\gamma=(\gamma_0,\dots,\gamma_i)$ is $\Z^d$- or $*$-connected
if for all $i \in [0,l-1]$, $\|\gamma_{i+1}-\gamma_i\|_2$ or $\|\gamma_{i+1}-\gamma_i\|_\infty$ equals $1$,
respectively.
Here$\| \cdot \|_2$ is the $\ell_2$-norm.
A $\Z^d$-connected path $\gamma=(\gamma_0,\dots,\gamma_i)$ is called bad
if each edge $\{ \gamma_i,\gamma_{i+1} \}$ is bad.
Furthermore, a $*$-connected path $\gamma=(\gamma_0,\dots,\gamma_i)$ is called unhealthy
if each site $\gamma_i$ is unhealthy.
Let $\mathcal{C}_- (x)$ be a bad $\Z^d$-connected cluster containing a site $x$,
i.e., the set of all sites connected to $x$ by a bad $\Z^d$-connected path.
We also denote by $\mathcal{C}_+ (x)$ an unhealthy $*$-connected cluster containing a site $x$,
i.e., the set of all sites connected to $x$ by an unhealthy $*$-connected path.

Fix $\delta<1/d$.
We now define a truncated weight $\sigma (e)$ as follows.
If one of the following conditions (1)--(3) holds,
then we set $\sigma(e):=\omega (e)$, otherwise $\sigma(e):=1$:
\begin{enumerate}
 \item $\kappa \leq \omega(e) \leq \kappa^{-1}$,
 \item $\omega(e)<\kappa$, and $e$ is connected to a bad $\Z^d$-connected cluster
       with less than $n^\delta$ vertices,
 \item $\omega (e)>\kappa^{-1}$, and $e$ is connected to an unhealthy $*$-connected cluster
       with less than $n^\delta$ vertices.
\end{enumerate}
Then, let $T_\sigma$ be the first passage time on the truncated weights $\sigma$.
Moreover, for $x \in \R^d$ and $n \geq 1$, let
\begin{align*}
 D_n(x):=x+\left[-3^d\kappa^{-1}n^\delta ,3^d\kappa^{-1}n^\delta \right]^d.
\end{align*}
We now consider the first passage time $T(D_n(0),D_n(n\xi))$ for each $\ell_2$-unit vector $\xi \in \R^d$.
Note that for all $x \in D_n(0)$ and $y \in D_n(n\xi )$,
\begin{align}
 T(D_n(0),D_n(n\xi ))
 \leq T(x,y) \leq T(D_n(0),D_n(n\xi ))
 +J_n(\xi,\omega),
 \label{eq:zhang_pas}
\end{align}
where $J_n(\xi,\omega)$ is the sum of $\omega (e)$
over all edges included in $D_n(0) \cup D_n(n\xi )$.

The following lemma is a minor modification of Lemma~8 and (3.23) in \cite{Zha10}.

\begin{lem}\label{lem:zhang}
We can choose $\kappa$ satisfying that,
for each $\ell_2$-unit vector $\xi\in\R^d$, there exist constants $\tilde{C}_1,\tilde{C}_2>0$
(which depend only on the law $\nu$, $\,d,\,\alpha,\,\delta$ and $\kappa$)
such that
\begin{align}
 P \bigl( T(D_n(0),D_n(n\xi ))\not= T_\sigma (D_n(0),D_n (n\xi )) \bigr)
 \leq \tilde{C}_1\exp \{ -\tilde{C}_2n^\delta \},
 \label{eq:zhang1}
\end{align}
and, for all $u>0$,
\begin{align}
\begin{split}
 &P\bigl( |T_\sigma (D_n(0),D_n(n\xi ))-E[T_\sigma (D_n(0),D_n(n\xi ))]|
               \geq u n^{1/2+3\delta}\bigr)\\
 &\leq \tilde{C}_1\exp \{ -\tilde{C}_2u^2n^\delta \}.
\end{split}
\label{eq:zhang2}
\end{align}
\end{lem}
\begin{proof}
We replace the component $(\log n)^{1+\delta}$ appearing in (1.10) of \cite{Zha10}
with $n^\delta$.
Then, the proofs of \eqref{eq:zhang1} and \eqref{eq:zhang2} follow from
the same strategy taken in \cite[Sections~2 and 3]{Zha10}, and we do not repeat it here.
As mentioned in Subsection~\ref{subsect:org},
an oversight is contained in the proof of Lemma~8 in \cite{Zha10}
and let us present a way to fix it.
In the beginning of its proof, the following claim is stated:
\begin{quote}
By Proposition~5.8 in \cite{Kes86_lect}, with a probability larger than $1-C_1\exp(-C_2n)$,
there exists an optimal path $\gamma$ for $T(D_n(0),D_n(nu))$ with $\# \gamma \leq Ln$.
\end{quote}
Because we now only assume $m_{\nu,\alpha}<\infty$,
this does not directly follow from Proposition~5.8 in \cite{Kes86_lect}.
To fix this problem, we replace the phrase ``$\# \gamma\leq Ln$''
with ``$\# \gamma \leq \exp\{Ln^\delta\}$''.
Let
\begin{align*}
 A_n:=
 \bigl\{
  \text{any optimal path $\gamma$ for $T(D_n(0),D_n(n\xi))$ satisfies $\# \gamma>\exp \{ Ln^\delta \}$}
 \bigr\}.
\end{align*}
Proposition~5.8 in \cite{Kes86_lect} then shows that
there are constants $C_1,\,C_2$ and $C_3$ such that
\begin{align*}
 &P\Bigl(
  \text{$\exists$ a path $\gamma$ from $0$ with $\# \gamma \geq \exp \{ Ln^\delta \}$
        but $T(\gamma)\leq C_1\exp\{Ln^\delta\}$} \Bigr)\\
 &\leq C_2\exp \bigl\{ -C_3\exp\{Ln^\delta\} \bigr\}.
\end{align*}
Chebyshev's inequality hence implies
\begin{align}\label{eq:fix}
\begin{split}
 P(A_n)
 &\leq C_2(\# D_n(0))\exp \bigl\{ -C_3\exp\{Ln^\delta\} \bigr\}\\
 &\quad
       +P\bigl( T(D_n(0),D_n(n\xi)) >C_1\exp\{Ln^\delta\} \bigr)\\
 &\leq C_2(\# D_n(0))\exp \bigl\{ -C_3\exp\{Ln^\delta\} \bigr\}
       +C_1^{-1}m_{\nu,1}n\exp\{-Ln^\delta\}\\
 &\leq C_4 \exp\{-C_5n^\delta\}
\end{split}
\end{align}
for some constants $C_4$ and $C_5$.
If
\begin{align*}
 T(D_n(0),D_n(n\xi)) \not= T_\sigma (D_n(0),D_n(n\xi)),
\end{align*}
then we have an edge $e \in \gamma$ satisfying
that $\# \mathcal{C}_-(v_e)>n^\delta$ or $\# \mathcal{C}_+(v_e)>n^\delta$,
where $v_e$ is an endpoint of the edge $e$.
Note that if $e \in \gamma$ with $\# \gamma \leq \exp\{Ln^\delta\}$,
then $v_e \in [-\exp\{Ln^\delta\},\exp\{Ln^\delta\}]^d$ holds.
Therefore, we have
\begin{align*}
 &P \bigl( T(D_n(0),D_n(n\xi ))\not= T_\sigma (D_n(0),D_n (n\xi )) \bigr)\\
 &\leq P(A_n)
       +\sum_{e \in [-\exp\{Ln^\delta\},\exp\{Ln^\delta\}]^d}
       P \bigl(
       \# \mathcal{C}_-(v_e)>n^\delta \text{ or } \# \mathcal{C}_+(v_e)>n^\delta
       \bigr).
\end{align*}
By the choice of $\kappa$, Theorem~6.1 of \cite{Gri99_book} implies that
there are constants $C_6$ and $C_7$ such that
the second term on the right-hand side is bounded above by
\begin{align*}
 \sum_{e \in [-\exp\{Ln^\delta\},\exp\{Ln^\delta\}]^d} 2\exp\{-C_6n^\delta \}
 \leq C_7 \exp\{dLn^\delta -C_6n^\delta \}.
\end{align*}
This, together with \eqref{eq:fix}, gives \eqref{eq:zhang1} for sufficiently small $L$.
\end{proof}

We need the following lemma to estimate the difference between
the expectations of $T$ and $T_\sigma$.

\begin{lem}\label{cor:zhang_cor}
For each $\ell_2$-unit vector $\xi \in \R^d$
there exist constants $\tilde{C}_3,\tilde{C}_4>0$
(which depend only on $\nu$, $d,\,\alpha,\,\delta$ and $\kappa$) such that
\begin{align*}
 \bigl| E[T(D_n(0),D_n(n\xi ))]-E[T_\sigma (D_n(0),D_n(n\xi ))] \bigr|
 \leq \tilde{C}_3n\exp \{ -\tilde{C}_4n^\delta \}.
\end{align*}
\end{lem}
\begin{proof}
Let $\Gamma:=\{T(D_n(0),D_n(n\xi ))\not= T_\sigma (D_n(0),D_n (n\xi ))\}$,
and set
\begin{align*}
 C_8:=\sqrt{d}\tilde{C}_1^{(\alpha-1)/\alpha} m_{\nu,\alpha}^{1/\alpha},\qquad
 C_9:=\tilde{C}_2(\alpha-1) /\alpha.
\end{align*}
Using H\"older's  inequality and \eqref{eq:zhang1}, we have
\begin{align*}
 E\left[T(D_n(0),D_n(n\xi ))\1{\Gamma} \right]
 &\leq \sqrt{d}\tilde{C}_1^{(\alpha -1)/\alpha}m_{\nu,\alpha}^{1/\alpha}
           n \exp \{  -n^\delta \tilde{C}_2(\alpha -1)/\alpha \}\\
 &= C_8n\exp \{ -C_9n^\delta \}.
\end{align*}
Therefore,
\begin{align*}
 E[T_\sigma (D_n(0),D_n(n\xi ))]+C_8n\exp \{ -C_9n^\delta \}
 \geq E[ T(D_n(0),D_n(n\xi ))].
\end{align*}
Similarly, since $\sigma(e) \leq \omega(e)+1$ holds for all $e \in \mathcal{E}$,
\begin{align*}
 E[T(D_n(0),D_n(n\xi ))]+C_{10}n\exp \{ -C_{11}n^\delta \}
 \geq E[ T_\sigma (D_n(0),D_n(n\xi ))]
\end{align*}
for some constants $C_{10}$ and $C_{11}$.
Thus, Lemma~\ref{cor:zhang_cor} follows
by choosing	$\tilde{C}_3:=C_8 \vee C_{10}$ and $\tilde{C}_4:=C_9 \wedge C_{11}$.
\end{proof}

In the next section, $\tilde{C}_i$'s are always constants appearing in this section.

\section{Proof of Theorem~\ref{thm:rate}} \label{sec:pf_rate}

\subsection{Kesten's approach}\label{subsect:kesap}
Let us first prepare some notations.
Fix an $\ell_2$-unit vector $\xi \in \R^d$,
and for $M \in \N$ let $U_1,\dots,U_K$ be all the vectors
with integer components and $\|U_k\|_\infty=M,\,1 \leq k \leq K$.
Define
\begin{align*}
 \Lambda (M,n) :=\min \biggl\{
 \sum_{k=1}^K p(k)E[T(0,U_k)] \biggr\} -n\mu (\xi ),
\end{align*}
where the minimum is over all choices of $p(k) \in \N_0$
such that
\begin{align}\label{eq:choice}
 \biggl\| \sum_{k=1}^K p(k)U_k-n\xi \biggr\|_\infty \leq M.
\end{align}

In \cite[pages~317--327]{Kes93}, the proof of \eqref{eq:kesnon} is composed of three steps.
The main parts are Steps~1 and 2 of \cite[pages~317--326]{Kes93}, so that
we will explain only these steps here.
Step~3 in \cite[pages~326--327]{Kes93} will be explained in the proof of Theorem~\ref{thm:rate}.

In Step~1 of \cite[page 317]{Kes93}, Kesten shows that
there exists a constant $C_1>m_{\nu,1}$ such that
for $M \in [n^{1/(d+1)},n]$ and $l \geq 1$,
\begin{align}
 l \Lambda (M,n)-C_1l M^{1/d}n^{(d-1)/d}
 \leq \Lambda (M,l n)
 \leq C_1l n.
 \label{eq:kesten}
\end{align}
His proof works under assumption \eqref{eq:lowmnt}.

In Step~2 of \cite[page~321]{Kes93},
it is proved that there are constants $c,c',C,C'>0$ such that
for large $n$ and $M$ as above and for $l \geq 2$,
\begin{align}\label{eq:kesgoal}
\begin{split}
 &P\biggl( a_{0,l n}(\xi) \leq l n\mu (\xi )+\frac{l}{2}\Lambda (M,n) \biggr)\\
 &\leq ce^{-ln}
       +\exp \biggl\{ c'\frac{ln}{M}\log M+ClM^{(2-d)/(2d)}n^{(d-1)/d}
       -C'\frac{l \Lambda (M,n)^2}{nM^{1/2}} \biggr\}.
\end{split}
\end{align}
We have to modify this estimate under assumption \eqref{eq:lowmnt}.
In particular, \eqref{eq:exp_ass} is required for bounds \eqref{eq:imp1} and \eqref{eq:imp2} below.
Thus, if \eqref{eq:lowmnt} is assumed instead of \eqref{eq:exp_ass},
then we must get a bound similar to \eqref{eq:kesgoal}
without \eqref{eq:imp2} and \eqref{eq:imp1}.
In fact, this is possible by replacing \eqref{eq:replaced} with Lemma~\ref{lem:step2},
which is proved in Subsection~\ref{subsect:fix}.

Let us give a sketch of Kesten's proof of \eqref{eq:kesgoal}.
Let $\gamma:=(v_0,v_1,\dots,v_p)$ be any self-avoiding nearest neighbor path
from $v_0=0$ to $v_p=[l n\xi]$ with passage time $T(\gamma) \leq l n\mu(\xi)+(l/2)\Lambda (M,n)$.
In addition, define the indices $\tau_0:=0$ and
\begin{align*}
 \tau_{i+1}:=\min \{ k \in (\tau_i ,p];\|v_k-v_{\tau_i}\|_\infty =M \},\qquad i \geq 0,
\end{align*}
with the convention $\min \emptyset =\infty$.
Set $Q:=\max\{ i \geq 0;\tau_i<\infty\} $
and $a_i:=v_{\tau_i}$ for $i \in[0,Q]$.
By definition of $Q$, we have
\begin{align*}
 \|v_k-v_{\tau_Q}\|_\infty <M,\qquad \tau_Q<k\leq p,
\end{align*}
and in particular,
\begin{align}
 \| v_{\tau_Q}-l n\xi\|_\infty
 \leq \| v_{\tau_Q}-v_p \|_\infty +\| [l n\xi ]-l n\xi \|_\infty
 \leq M.
 \label{eq:inbox}
\end{align}
Moreover,
\begin{align}
 \| a_i-a_{i-1} \|_\infty
 = \| v_{\tau_i}-v_{\tau_{i-1}} \|_\infty
 = M,\qquad 1 \leq i \leq Q,
 \label{eq:bdbox}
\end{align}
so that $a_i-a_{i-1}$ is one of the $U_k$'s
(which appear in the beginning of this section).
It holds from \cite[pages~322--323]{Kes93} that there exists constants $C_2,C_3$ such that
\begin{align}
 P(Q \geq C_2ln/M) \leq C_3e^{-ln}.
 \label{eq:path_length}
\end{align}
We now fix $Q<C_2l n/M$
and $a_1,\dots,a_Q$ satisfying \eqref{eq:inbox} and \eqref{eq:bdbox}.
We denote by $p(k)$ the number of $i \in [1,Q]$ with $a_i-a_{i-1}=U_k$.
The $p(k)$'s are fixed at the moment.
Then, (3.28)--(3.32) of \cite[page~323]{Kes93} enable us to show that for any $\beta \geq 0$,
\begin{align}
\begin{split}
 &P \biggl( \begin{minipage}{24em}
             $\exists$ a self-avoiding path $\gamma$ with
             $v_{\tau_i}=a_i,\,1 \leq i \leq Q$,\\
             and satisfying \eqref{eq:inbox}
             and $T(\gamma) \leq l n\mu(\xi)+(l/2)\Lambda (M,n)$
            \end{minipage}
            \biggr)\\
 &\leq \exp \biggl\{ -\frac{\beta l}{2}\Lambda (M,n) +\beta C_1 l M^{1/d}n^{(d-1)/d} \biggr\}\\
 &\quad \times \prod_{k=1}^K E \bigl[ \exp \bigl\{-\beta (T(0,U_k)-E[T(0,U_k)]) \bigr\} \bigr]^{p(k)}.
\end{split}
\label{eq:path}
\end{align}
It remains to estimate the product in \eqref{eq:path}.
Note that $\sum_{k=1}^K p(k)=Q$, which is the number of $(a_i-a_{i-1})$'s, and
\begin{align}
\begin{split}
 &E \bigl[ \exp \bigl\{-\beta (T(0,U_k)-E[T(0,U_k)]) \bigr\} \bigr]\\
 &\leq \exp \biggl\{ C_4 \frac{\beta l}{Q}\Lambda (M,n) \biggr\}\\
 &\quad +\exp \{ \beta E[T(0,U_k)] \}
        P \biggl( T(0,U_k)-E[T(0,U_k)] \leq -\frac{C_4l}{Q}\Lambda (M,n) \biggr),
\end{split}
\label{eq:approach}
\end{align}
where $C_4$ will be chosen such that for large $M$ and for $n \geq M$ and $l \geq 2d$,
\begin{align}
 \frac{C_4l}{Q}\Lambda (M,n)
 \leq \frac{d}{2}Mm_{\nu,1}
 \quad \textnormal{and}\quad C_4 \leq \frac{1}{4}.
 \label{eq:constbd}
\end{align}
The argument below (3.34) of \cite{Kes93}
guarantees the existence of such a $C_4$.
In particular, for $n \geq M$ and $l \geq 2d$,
\begin{align}
 Q \geq \frac{l n}{dM}-1 \geq \frac{ln}{2dM}.
 \label{eq:q_lower}
\end{align}
We shall estimate the last probability in \eqref{eq:approach}.
Set $\eta:=U_k/\|U_k\|_2$ and $m:=\lfloor \|U_k\|_2 \rfloor \in [M,dM]$.
Note that $\|[m\eta] -U_k\|_\infty \leq 2$.
Assumption \eqref{eq:exp_ass} guarantees that there exist constants $c,C,C'>0$ such that for $t \geq 0$,
\begin{align}\label{eq:imp2}
 P\bigl( \bigl| T(0,[m\eta ])-E[T(0,[m\eta ])] \bigr| \geq t\sqrt{m} \bigr)
 \leq Ce^{-C't},
\end{align}
and for $t \leq cm$,
\begin{align}\label{eq:imp1}
 P \bigl( \bigl| T(0,[m\eta ])-E[T(0,[m\eta ])]-T(0,U_k)+E[T(0,U_k)] \bigr| \geq t \bigr)
 \leq Ce^{-C't},
\end{align}
which are (2.49) and (3.36) of \cite{Kes93}, respectively.
By choosing $t$ suitably (see (3.37) of \cite[page~325]{Kes93} for details), these estimates show
that for some constants $C_5,C_6>0$,
\begin{align}\label{eq:replaced}
\begin{split}
 &P \biggl( T(0,U_k)-E[T(0,U_k)] \leq -\frac{C_4l}{Q}\Lambda (M,n) \biggr)\\
 &\leq C_5\exp \biggl\{ -\frac{C_6}{QM^{1/2}}l\Lambda (M,n) \biggr\}.
\end{split}
\end{align}
Therefore, the right-hand side of \eqref{eq:approach} is at most
\begin{align*}
 \exp\left\{ C_4\frac{\beta l}{Q}\Lambda(M,n)\right\}
 +C_7\exp\biggl\{\beta dMm_{\nu,1}-\frac{C_6}{QM^{1/2}} l\Lambda(M,n) \biggr\}.
\end{align*}
for some constant $C_7$.
Choose $\beta$ such that the two exponents become equal, so that
the left-hand side of \eqref{eq:approach} is smaller than
\begin{align*}
 C_8^Q
 \exp \biggl\{ C_9l M^{\delta -(d-2)/(2d)}n^{(d-1)/d}
               -C_{10}\frac{l^2\Lambda(M,n)^2}{QM^{3/2}}\biggr\}.
\end{align*}
for some constants $C_8,C_9,C_{10}$.
Hence \eqref{eq:kesgoal} follows by
summing the left-hand side of \eqref{eq:path} over all possible values of $Q$ and $a_1,\dots,a_Q$.
(See the first paragraph of \cite[page 326]{Kes93} for details.)

With these observations, under \eqref{eq:lowmnt} we must estimate the last probability in \eqref{eq:approach}
without \eqref{eq:imp2} and \eqref{eq:imp1}.
In fact, this is possible as follows.
(See Subsection~3.2 for the proof.)

\begin{lem}\label{lem:step2}
Assume \eqref{eq:p_c} and \eqref{eq:lowmnt}.
For $\delta \leq1/6$
there exist constants $C_{11},C_{12}>0$ such that,
for all large $n$, if $\Lambda (M,n) \geq C_{11}nM^{-(1-d\delta )}$ and $Q<C_2ln/M$, then
\begin{align*}
 &P\biggl( T(0,U_k)-E[T(0,U_k)]\leq -\frac{C_4l}{Q}\Lambda(M,n)\biggr)\\
 &\leq 2\tilde{C}_1
       \exp\biggl\{ -C_{12}M^{-(2-\delta )} \biggl( \frac{l}{Q} \biggr)^2 \Lambda (M,n)^2 \biggr\}.
\end{align*}
\end{lem}

\subsection{Proofs of Lemma~\ref{lem:step2} and Theorem~\ref{thm:rate}}\label{subsect:fix}
Let us first give the proof of Lemma~\ref{lem:step2}.

\begin{proof}[\bf Proof of Lemma~\ref{lem:step2}]
Recall that $\eta:=U_k/\|U_k\|_2$ and $m:=\lfloor \|U_k\|_2 \rfloor \in [M,dM]$.
Note that
$\|m\eta -U_k\|_\infty \leq 1$
and $0 \in D_{m}(0)$ and $U_k \in D_{m}(m\eta)$
hold for large $m$.
By \eqref{eq:zhang_pas},
\begin{align*}
 T(D_{m}(0),D_{m}(m\eta))
 \leq T(0,U_k)
 \leq T(D_{m}(0),D_{m}(m\eta))
      +J_m(\xi,\omega).
\end{align*}
This, together with Lemma~\ref{cor:zhang_cor}, gives
\begin{align*}
 E[T(0,U_k)]
 \leq E[T_\sigma (D_{m}(0),D_{m}(m\eta ))]
         +\tilde{C}_3m\exp \{ -\tilde{C}_4m^\delta \} +C_{13}m^{d\delta}
\end{align*}
for some constant $C_{13}$.
Therefore,
\begin{align}
\begin{split}
 &P\biggl( T(0,U_k)-E[T(0,U_k)]\leq -\frac{C_4l}{Q}\Lambda(M,n)\biggr)\\
 &\leq P\biggl( T(D_{m}(0),D_{m}(m\eta))-E[T_\sigma (D_{m}(0),D_{m}(m\eta))]\\
 &\qquad\quad
               \leq -\frac{C_4l}{Q}\Lambda(M,n)
                    +\tilde{C}_3m\exp\{-\tilde{C}_4m^\delta\}
                    +C_{13}m^{d\delta}\biggr).
\end{split}
\label{eq:prob_main}
\end{align}
Take $C_{11}:=4d^{d\delta}C_2(\tilde{C}_3 \vee C_{13})/C_4$.
Since we have assumed $\Lambda (M,n) \geq C_{11}nM^{-(1-d\delta)}$ and $Q<C_2ln/M$,
the choice of $n,\,M$ and $m$ implies for all large $n$,
\begin{align*}
 \frac{C_4l}{2Q}\Lambda (M,n)
 \geq \tilde{C}_3m\exp \{-\tilde{C}_4m^\delta \} +C_{13}m^{d\delta}.
\end{align*}
It follows that the right-hand side of \eqref{eq:prob_main} is smaller than
\begin{align*}
 P\biggl( T(D_{m}(0),D_{m}(m\eta))-E[T_\sigma (D_{m}(0),D_{m}(m\eta ))]
            \leq -\frac{C_4l}{2Q}\Lambda(M,n) \biggr).
\end{align*}
Thanks to \eqref{eq:zhang1} and \eqref{eq:zhang2},
this is bounded from above by
\begin{align*}
 &\tilde{C}_1\exp\{ -\tilde{C}_2m^\delta \}\\
 &+P\biggl( |T_\sigma (D_{m}(0),D_{m}(m\eta))-E[T_\sigma (D_{m}(0),D_{m}(m\eta))]|
                 \geq \frac{C_4l}{2Q}\Lambda(M,n) \biggr)\\
 &\leq \tilde{C}_1\exp\{ -\tilde{C}_2m^\delta \}
       +\tilde{C}_1\exp\biggl\{ -\biggl( \frac{\tilde{C}_2C_4^2}{4} \biggr)
                              \frac{(l /Q)^2\Lambda(M,n)^2}{m^{1+5\delta}} \biggr\}.
\end{align*}
By \eqref{eq:constbd} and $\delta \leq 1/6$, there exists a constant $C_{12}>0$
such that the right-hand side is smaller than
\begin{align*}
 2\tilde{C}_1
 \exp\biggl\{ -C_{12}M^{-(2-\delta )} \biggl( \frac{l}{Q} \biggr)^2 \Lambda (M,n)^2 \biggr\}.
\end{align*}
Hence the proof is complete.
\end{proof}

Finally, we prove Theorem~\ref{thm:rate}.

\begin{proof}[\bf Proof of Theorem~\ref{thm:rate}]
Let us first show that
there exist constants $C_{14},C_{15},C_{16}>0$ such that,
for all large $n$, if $\Lambda (M,n) \geq C_{11}nM^{-(1-d\delta )}$, then
\begin{align}
\begin{split}
 &P\biggl( a_{0,l n}(\xi) \leq l n\mu (\xi )+\frac{l}{2}\Lambda (M,n) \biggr)\\
 &\leq C_{14}e^{-l n}
           +\exp \biggl\{ C_{15}\frac{ln}{M}\log M+C_{15}l M^{\delta -(d-1)/d}n^{(d-1)/d}
                                  -C_{16}\frac{l \Lambda (M,n)^3}{n^2M^{1-\delta}} \biggr\},
\end{split}
\label{eq:step2}
\end{align}
which is the counterpart of \eqref{eq:kesgoal} under \eqref{eq:lowmnt}.
From Lemma~\ref{lem:step2}, the right-hand side of \eqref{eq:approach} is at most
\begin{align*}
 \exp\left\{ C_4\frac{\beta l}{Q}\Lambda(M,n)\right\}
 +2\tilde{C}_1
      \exp\biggl\{\beta dMm_{\nu,1}-C_{12}M^{-(2-\delta )}
      \biggl( \frac{l}{Q} \biggr)^2 \Lambda(M,n)^2 \biggr\}.
\end{align*}
Finally, we choose $\beta$ such that the two exponents above here become equal, i.e.,
\begin{align*}
 \beta
 = C_{12}M^{-(2-\delta )}
      \biggl( \frac{l}{Q} \biggr)^2 \Lambda(M,n)^2
      \biggl( dMm_{\nu,1}-\frac{C_4l}{Q}\Lambda(M,n) \biggr)^{-1}.
\end{align*}
In particular, by \eqref{eq:constbd},
\begin{align*}
 \beta
 \leq C_{12}M^{-(2-\delta )}
         \biggl( \frac{l}{Q} \biggr)^2 \Lambda(M,n)^2
         \biggl( \frac{d}{2}Mm_{\nu,1} \biggr)^{-1}
 \leq C_{17}M^{-(1-\delta)}
\end{align*}
for some constant $C_{17}$.
By \eqref{eq:constbd} and \eqref{eq:q_lower},
the left-hand side of \eqref{eq:path} is smaller than
\begin{align*}
 &\exp\{\beta C_1l M^{1/d}n^{(d-1)/d}\}
  \times\prod_{k=1}^K
  \biggl( (2\tilde{C}_1+1)\exp\biggl\{\biggl( C_4-\half \biggr)
  \frac{\beta l}{Q}\Lambda(M,n) \biggr\}\biggr)^{p(k)}\\
 &\leq (2\tilde{C}_1+1)^{C_2l n/M}
           \exp \biggl\{ C_1C_{17}l M^{\delta -(d-1)/d}n^{(d-1)/d}
                               -C_{18}\frac{l\Lambda(M,n)^3}{n^2M^{1-\delta}}\biggr\}
\end{align*}
for some constant $C_{18}$.
Therefore,
bound \eqref{eq:step2} follows by
summing the left-hand side of \eqref{eq:path} over all possible values
of $Q$ and $a_1,\dots,a_Q$. 
See the first paragraph in \cite[page~326]{Kes93} for details.

We complete the proof of Theorem~\ref{thm:rate} following basically Step~3 of \cite[pages~326--327]{Kes93}.
Pick
\begin{align}
 \delta :=1/(d+4) .
 \label{eq:delta}
\end{align}
Here, note that $\delta <1/d$.
We first treat the case $\Lambda (M,n) \geq C_{11}nM^{-(1-d\delta )}$.
Choose
\begin{align}
 M:=\lfloor n^{1/(d\delta+1)} \rfloor.
 \label{eq:m}
\end{align}
If we have
\begin{align}\label{eq:contra}
 C_{16}\frac{l \Lambda (M,n)^3}{n^2M^{1-\delta}}
 >C_{15}\frac{ln}{M}\log M+C_{15}l M^{\delta -(d-1)/d}n^{(d-1)/d},
\end{align}
then by \eqref{eq:step2},
\begin{align*}
 \lim_{l \to \infty} P\left( a_{0,l n}(\xi) \leq l n\mu (\xi )+\frac{l}{2}\Lambda (M,n) \right)
 = 0.
\end{align*}
However, this contradicts to \eqref{eq:timeconst}, and \eqref{eq:contra} fails to hold.
This means that
\begin{align*}
 \Lambda (M,n)
 \leq C_{19}\left\{ nM^{-\delta /3}(\log M)^{1/3}+n^{1-1/(3d)}M^{1/(3d)} \right\}
\end{align*}
for some constant $C_{19}$.
By \eqref{eq:delta}, $\Lambda (M,n)$ is smaller than
\begin{align*}
 2C_{19}nM^{-\delta /3}(\log M)^{1/3}
 \leq C_{20}n^{1-1/(6d+12)}(\log n)^{1/3}
\end{align*}
for some constant $C_{20}$.
This, together with the definition of $\Lambda (M,n)$, enables us to
take $p(k) \geq 0$ satisfying \eqref{eq:choice} and
\begin{align*}
 \sum_{k=1}^\nu p(k)E[T(0,U_k)]
 \leq n \mu (\xi )+C_{20}n^{1-1/(6d+12)}(\log n)^{1/3}.
\end{align*}
Now set $\rho =\sum_{k=1}^\nu p(k)$
and let $u_1,\dots,u_\rho$ be the sites defined by
$u_i-u_{i-1}=U_k$ for $\sum_{j=1}^{k-1} p(j)<i\leq \sum_{j=1}^k p(j)$.
Note that $u_\rho =\sum_{k=1}^\nu p(k)U_k$.
Subadditivity of the first passage time gives
\begin{align*}
 E[a_{0,n}(\xi )]
 \leq \sum_{i=1}^\rho E[T(u_{i-1},u_i)]+E[T(u_\rho ,n\xi )].
\end{align*}
By the choice of $u_1,\dots,u_\rho$,
\begin{align*}
 \sum_{i=1}^\rho E[T(u_{i-1},u_i)]
 &= \sum_{k=1}^\nu p(k)E[T(0,U_k)]\\
 &\leq n\mu (\xi )+C_{20}n^{1-1/(6d+12)}(\log n)^{1/3}.
\end{align*}
In addition, by \eqref{eq:choice},
\begin{align*}
 E[T(u_\rho ,n\xi )]
 \leq d\|[n\xi ]-u_\rho \|_\infty E[\omega (0)]
 \leq d(M+1)E[\omega (0)],
\end{align*}
and \eqref{eq:rate} immediately follows in the case $\Lambda (M,n) \geq C_{11}nM^{-(1-d\delta )}$.

In the case $\Lambda (M,n)<C_{11}nM^{-(1-d\delta )}$,
the definition of $\Lambda(M,n)$ implies
\begin{align*}
 n\mu (\xi )+C_{11}nM^{-(1-d\delta )}
 &> \min \biggl\{
 \sum_{k=1}^K p(k)E[T(0,U_k)] \biggr\},
\end{align*}
where the minimum is taken over all choices of $p(k)$ satisfying \eqref{eq:choice}.
Subadditivity of the first passage time shows that
\begin{align*}
 \sum_{k=1}^K p(k)E[T(0,U_k)]
 &\geq \sum_{k=1}^K E \Biggl[ T \Biggl( \sum_{j=1}^{k-1}p(j)U_j, \sum_{j=1}^kp(j)U_j \Biggr) \Biggr]\\
 &\geq -d(M+1)m_{\nu,1}+E[a_{0,n}(\xi )].
\end{align*}
With these observations,
\begin{align*}
 E[a_{0,n}(\xi )]
 \leq n\mu (\xi )+C_{11}nM^{-(1-d\delta )}+d(M+1)m_{\nu,1}.
\end{align*}
This, together with \eqref{eq:delta} and \eqref{eq:m}, is bounded from above by
\begin{align*}
 n\mu (\xi )+(C_{11}+2dm_{\nu,1})n^{1-1/(d+2)}(\log n)^{1/3}.
\end{align*}
Since $n^{1-1/(6d+12)} \geq n^{1-1/(d+2)}$, \eqref{eq:rate} is valid in all cases.
\end{proof}

\paragraph{\bf Acknowledgments}
I am grateful to Michael Damron for discussions on this problem.
The author would also like to express his gratitude to the reviewer
for the careful reading of the manuscript.


\end{document}